\documentclass[10pt]{amsart}
\usepackage{amsfonts}
\usepackage{amsmath}
\usepackage{amssymb}
\usepackage{graphicx}
\usepackage{enumerate}
\usepackage{multicol}
\usepackage{mathrsfs}
\usepackage[usenames,dvipsnames]{pstricks}
\usepackage{epsfig}
\usepackage{pst-grad} % For gradients
\usepackage{pst-plot} % For axes

\newtheorem{theorem}{Theorem}
\newtheorem{corollary}{Corollary}
\newtheorem{lemma}{Lemma}
\newtheorem{definition}{Definition}
\newtheorem{proposition}{Proposition}
\theoremstyle{remark}
\newtheorem{remark}{Remark}
\newtheorem{example}{Example}

\newcommand{\real}{{\mathbb{R}}}

\DeclareMathOperator{\SO}{SO}

\DeclareMathOperator{\SU}{SU}
\DeclareMathOperator{\GL}{GL}

\DeclareMathOperator{\Isom}{Isom}
\DeclareMathOperator{\pr}{pr}
\DeclareMathOperator{\rank}{rank}

\title{Geometric conditions for the existence of an intrinsic rolling}
\author[Erlend Grong, Mauricio Godoy M.]{Mauricio Godoy Molina \\ Erlend Grong}

\address{CMAP, \'Ecole Polytechnique, CNRS, 91128 Palaiseau, France.}
\email{mauricio.godoy@cmap.polytechnique.fr}

\address{Department of Mathematics, University of Bergen, Norway.}
\email{erlend.grong@math.uib.no}



\begin{document}
\maketitle

\begin{abstract}
We give a complete answer to the question of when two curves in two different Riemannian manifolds can be seen as trajectories of rolling one manifold on the other without twisting or slipping. We show that up to technical hypotheses, a rolling along these curves exists if and only if the geodesic curvatures of each curve coincide. By using the anti-developments of the curves, which we claim can seen of as a generalization of the geodesic curvatures, we are able to extend the result to arbitrary absolutely continuous curves. For a manifold of constant sectional curvature rolling on itself, two such curves can only differ by an isometry. In the case of surfaces, we give conditions for when loops in the manifolds lift to loops in the configuration space of the rolling.
\end{abstract}

\section{Introduction}

Rolling of surfaces without slipping or twisting is one of the classical kinematic problems that in recent years has again attracted the attention of mathematicians due to its geometric and analytic richness. A very interesting historical account of problems in non-holonomic dynamics can be found in~\cite{BMZ}, in which the problem of the sphere rolling on the plane is presented as one of the earliest examples of a non-holonomic mechanical system. The interest in this particular case can be traced back as far as the late 19th century, for instance, see~\cite{Chap1,Chap2}. A detailed exposition of the non-holonomy of the rolling sphere is presented in~\cite{J}.

The definition of the so-called rolling map, which corresponds to rolling manifolds of dimension higher than two imbedded in ${\mathbb{R}}^m$ without slipping or twisting, was given for the first time in~\cite{Sharpe}. This was the starting point of~\cite{GGMS} where this extrinsic point of view was shown to be equivalent to a purely intrinsic condition and a condition depending solely on the imbeddings of the manifolds. The extrinsic point of view, which depends on the imbeddings, has been successfully applied in some particular cases, obtaining interpolation results~\cite{HS} and controllability~\cite{JZ, Zimm}.
In this work we address the problem of existence of rollings for two abstract Riemannian manifolds. We employ the coordinate-free approach introduced in~\cite{GGMS} which allows us to consider the problem with purely intrinsic methods.

From \cite{Sharpe}, it has been know that any rolling of one manifold $M$ on the other $\widehat M$ without slipping or twisting, is completely determined by the curve $x(t)$ it rolls along in $M$ up to initial configuration. A natural question to ask is if the curve $\widehat x(t)$ in $\widehat M$ produced by the rolling is completely determined by $x(t)$, how does geometry of $\widehat x(t)$ compare to that of $x(t)$? Is there any way in which we, just by studying the curves $x(t)$ and $\widehat x(t)$ can determine if there exist a rolling along this curves? We will present a complete answer to these question in this paper.

For a long time, mathematicians have had the intuition that by rolling an $n$ dimensional manifold $M$ along a given curve $y(t)$ in $\real^n$ with the Euclidean structure, one would obtain a curve in $M$ which resembles the original curve $y(t)$. This seems to be the reason why the Brownian motion on manifolds is defined by a rolling following a Brownian motion in ${\mathbb R}^n$, see~\cite[Chapter 2]{Hsu}, or why rolling is used in interpolation, see~\cite{HS}. We will formalize this by trying to answer the following question: {\em Given a curve $x(t)$ in $M$, how can you represent it by a curve $\widehat x(t)$ in another Riemannian manifold $\widehat M$, keeping as much local information as possible?} An answer to this question already exists for some particular classes of curves. If $x(t)$ is a geodesic then, at least for short time, we can define $\widehat x(t)$ by $\widehat x(t) = \exp_{\widehat x(0)} \circ q_0 \circ \exp_{x(0)}^{-1}(x(t))$, where $q_0: T_{x(0)}M \to T_{\widehat x(0)} \widehat M$ is a chosen isometry and $\exp$ denotes the Riemannian exponential map. If $x(t)$ has $n-1$ well defined geodesic curvatures, then this data is enough to construct a curve $\widehat x(t)$ in $\widehat M$ with the same curvatures. However, if $x(t)$ is a general absolutely continuous curve, which might only have derivatives defined almost everywhere, it is neither clear how to construct $\widehat x(t)$, nor what would the term ``local data'' means.

We would like to make the argument that the best way to obtain a curve $\widehat x(t)$ in $\widehat M$ that preserves ``local data'', is to roll $M$ on $\widehat M$ without twisting or slipping. We will also show that obtaining a curve in this way, can be seen as a generalization of the techniques in the previously mentioned particular cases.

\subsection{Structure of the paper and main results}

In Section~\ref{sec:prelim2dim} we give the definition of intrinsic rolling and anti-development, both of which will be used throughout this paper. In this section we also treat the special case of dimension 2. We are able to deduce two important results for surfaces: the existence of rollings in terms of the geodesic curvature and when a loop in a manifold lifts to a loop in the configuration space. To explain in more detail what we mean by the latter, we try to answer the following question: Given a loop $x(t)$ on a surface $M$, if we roll $M$ on $\real^2$ along $x(t)$, when will we, after the rolling is completed, have returned back to the initial configuration? We are able to answer this only in terms of geometric properties of $x(t)$.

In higher dimensions the situation is more subtle and requires the use of heavier differential geometric machinery. In Section~\ref{sec:rolldimn} we therefore introduce the concepts of Frenet frames and geodesic curvatures and prove the existence of rollings in arbitrary dimensions in terms of these notions. We also address the question of existence for general absolutely continuous curves, where it may happen that the concepts mentioned before are not well defined. We are able to give an answer in terms of anti-development curves, which we argue generalizes the concept of geodesic curvatures.

After describing the theoretical benefits of rolling without twisting or slipping, we end with a more practical problem. The equations governing a rolling motion, are not in general easy to solve explicitly. However, we can somewhat simplify the problem if we know that one of the curves has all geodesic curvatures well defined. We will make some brief comments on this in Section~\ref{sec:construction}.

\vspace{0.5cm}
\noindent {\bf Acknowledgments:} We thank Georgy Ivanov, Irina Markina and Martin Stolz for their willingness to discuss some of the results presented here, as well for their useful suggestions and remarks. We also thank professor Richard Montgomery for proposing the problem in Section \ref{sec:dim2} to us and for some very enlightening conversations.

\section{Preliminaries and the two dimensional case}\label{sec:prelim2dim}

\subsection{Intrinsic rolling} \label{subsec:recalling}

The aim of this section is to provide the necessary background and notations of the coordinate-free approach of rolling manifolds without slipping or twisting as presented in~\cite{GGMS}. As customary, in the rest of the article we use simply {\em rolling} to refer to rolling without slipping or twisting.

Let $M$ and $\widehat M$ be two connected, oriented Riemannian manifolds of dimension~$n$. The configuration space $Q$ for rolling is the ${\rm SO}(n)-$bundle
\begin{equation}\label{defQ}
Q = \left\{ q \in \SO(T_xM,T_{\widehat{x}} \widehat{M}) \, \colon \, x \in M, \, \widehat{x} \in \widehat{M} \right\},
\end{equation}
where $\SO(V,W)$ stands for the space of linear, positively oriented isometries of oriented inner-product spaces $V$ and $W$. As noted in~\cite{GGMS}, the bundle $Q$ can also be represented as
\[Q=(F(M)\times F({\widehat M}))/{\rm SO}\,(n),\]
where $F(M)$ denotes the oriented orthonormal frame bundle of $M$, i.e. the principal ${\rm SO}(n)-$bundle where the fiber over a point $x\in M$ consist of the positively oriented orthonormal frames in $T_{x}M$, and the quotient is with respect to the diagonal ${\rm SO}(n)-$action.

\begin{remark}
The ${\rm SO}(n)-$bundle $Q$ is a principal bundle in the general situation {\em only } when $n=2$. In other words, if $n\geq3$ it is always possible to find manifolds such that $Q$ is not principal, see~\cite[Proposition 3.4]{Chitour}.
\end{remark}

Denoting by $\pi \colon Q\to M$ the projection onto $M$ and similarly the projection $\widehat \pi$~to~$\widehat M$, we can state the definition of an intrinsic rolling.

\begin{definition} \label{intrinsdef}
An intrinsic rolling of $M$ on $\widehat{M}$ is an absolutely continuous curve $q\colon [0,\tau] \rightarrow Q$, satisfying the following conditions: if $x(t) = \pi q(t)$ and $\widehat{x}(t) = \widehat \pi q(t)$, then
\begin{itemize}
\item[(I)] no slip condition: $\dot{\widehat{x}}(t) = q(t) \dot{x}(t)$ for almost all~$t$;
\item[(II)] no twist condition: $\displaystyle{q(t) \frac{D}{dt}\, Z(t) = \frac{D}{dt}\, q(t) Z(t)}$
for any vector field $Z(t)$ on $M$ along $x(t)$ and almost every $t$.
\end{itemize}
\end{definition}

The symbol $\frac{D}{dt}$ stands for the covariant derivative of the Levi-Civita connection on the respective manifolds. The main result in~\cite{GGMS} (Theorem 2) states that given an intrinsic rolling $q(t)$, isometric imbeddings of $M$ and ${\widehat M}$ into a common Euclidean space ${\mathbb R}^{N}$, and an initial configuration of the imbedded manifolds, there is a unique rolling in the sense of Sharpe~\cite[Appendix B]{Sharpe} yielding to the same dynamics as the original intrinsic rolling $q(t)$.

In the following sections, the letter $Q$ will always denote the configuration space of the intrinsic rolling for the manifolds under consideration, and it will always be considered as the bundle of isometries~\eqref{defQ}.
Moreover, we assume that all the manifolds are connected, oriented and Riemannian. The space ${\mathbb R}^n$ will always come furnished with the standard Euclidean structure.

\subsection{Existence of rolling in two dimensions}

If we have two $n$ dimensional manifolds $M$ and $\widehat M$ and a given absolutely continuous curve $x(t)$ in $M$, there is a rolling $q(t)$ of $M$ on $\widehat M$, such that $\pi q(t) = x(t)$, for short time in general, and for all time if $\widehat M$ is complete, see \cite[Lemma 6]{Grong}. It is uniquely determined by an initial configuration $q(0): T_{x(0)} M \to T_{\widehat x_0} \widehat M$, for some $\widehat x_0 \in \widehat M$. We say that $q(t)$ is a rolling of $M$ on $\widehat M$ along $x(t)$. Given such a rolling, it is clear that $\widehat x(t) = \widehat \pi q(t)$ is completely determined by $x(t)$ and $q(0)$, but it is not immediately clear how. We therefore seek to understand what conditions a pair of curves $(x(t), \widehat x(t))$ must satisfy in order for there to exist a rolling $q(t)$ along these curves, i.e. with $\pi q(t) = x(t)$ and $\widehat \pi q(t) = \widehat x(t).$
Before trying to give sufficient conditions for the general situation, let us see the concrete case of surfaces. In what follows, by a surface we mean a 2-dimensional, connected and oriented Riemannian manifold.

It is clear from the no slip condition, that requiring $x(t)$ and $\widehat x(t)$ have the same length is a necessary condition for the existence of a rolling $q(t)$. It is easy to construct examples to see that this is not sufficient.
Let us start by letting both $x(t)$ and $\widehat x(t)$ be $C^2$ curves, both parametrized by arc-length. In this case, the problem of existence of a rolling following given trajectories has a complete solution.

\begin{definition}\label{def:geodcurv2}
Given a $C^2$ curve $x(t)$ in a surface $M$, that is parametrized by arc-length, write $\nu(t)$ for the unique vector field along $x(t)$ such that $\{\dot x(t), \nu(t) \}$ is a positively oriented orthonormal basis for every $t$. Then $k_g(t) = \left\langle \tfrac{D}{dt} \dot x(t) , \nu(t)\right\rangle$ is called the oriented geodesic curvature of $x(t)$.
\end{definition}

\begin{theorem}\label{geodcurv}
Let $M$ and $\widehat M$ be surfaces. Let $x\colon [0,\tau]\to M$ and $\widehat x\colon [0,\tau]\to\widehat M$ be two $C^2$ curves, parameterized by arc-length. Then, there is a rolling $q(t)$ such that
$$\pi q(t) = x(t), \qquad \widehat \pi q(t) = \widehat x(t)$$
if and only if the oriented geodesic curvatures of $x(t)$ and $\widehat x(t)$ coincide.
\end{theorem}

\begin{proof}
Choose vector fields $\nu(t)$ along $x(t)$ and $\widehat \nu(t)$ along $\widehat x(t)$ as in Definition~\ref{def:geodcurv2} and denote by $k_g$ and $\widehat k_g$ the respective oriented geodesic curvatures. Note that, by definition, $\nu(t)$ and $\widehat \nu(t)$ are differentiable at least once. It is clear that the unique positively oriented isometry $q(t): T_{x(t)}M \to T_{\widehat x(t)} \widehat M$ which satisfy the no slip condition is the isometry determined by
\begin{eqnarray*}
q(t) \dot x(t) &=& \dot{\widehat x}(t) ,\\
q(t) \nu(t) &=& \widehat \nu(t),
\end{eqnarray*}
for all $t$. We need to study when does $q(t)$ also satisfy the no twist condition. In order to do that, it is enough to show that it holds for $\dot x(t)$ and $\nu(t)$.
From the equations
\begin{eqnarray*}
q(t) \frac{D}{dt}\, \dot x(t) = k_g(t) \widehat \nu(t), && \frac{D}{dt}\, q(t) \dot x(t) =  \widehat k_g(t) \widehat \nu(t),\\
q(t) \frac{D}{dt}\, \nu(t) = -k_g(t) \dot{\widehat x}(t), && \frac{D}{dt}\, q(t) \nu(t) = -\widehat k_g(t) \dot{\widehat x}(t),
\end{eqnarray*}
it follows that the no twist condition holds if and only if $k_g(t) = \widehat k_g(t)$ for any~$t$.
\end{proof}
In order to interpret this theorem in more detailed a manner, we will introduce the notions of development and anti-development. We will discuss this in general for $n$-dimensional manifolds.

\subsection{Anti-development}\label{subsec:dev}

A general frame at $x\in M$ is a fixed linear isomorphism $f\colon \real^n\to T_xM$. Each general frame $f$ gives a basis $f_1, \dots, f_n$ of $T_xM$, defined by $f_j:=f(e_j)$, where ${\{e_j\}}^n_{j=1}$ is the canonical ordered basis of ${\mathbb R}^n$. Denote the set of all general frames at $x$ by ${\mathcal F}_x(M)$. The general frame bundle ${\mathcal F}(M)=\coprod_x{\mathcal F}_x(M)$ can naturally be given the structure of a manifold of dimension $n(n+1)$ with a principal $\GL(n, {\mathbb R})-$structure. The manifold structure of ${\mathcal F}(M)$ is such that the natural projection $\pr_{M}\colon {\mathcal F}(M)\to M$ is a smooth map.

Let $M$ be equipped with an affine connection $\nabla$. A curve $f\colon [0,\tau]\to{\mathcal F}(M)$ is called horizontal if the vector fields $f_j(t)$ are parallel along $\pr_{M} \circ f\colon [0,\tau]\to M$ with respect to $\nabla$. The set of tangent vectors of all horizontal curves forms an $n$-dimensional distribution $E$ over ${\mathcal F}(M)$ called the Ehresmann connection associated to $\nabla.$ For any point $f \in \mathcal F(M)$, a vector $v \in E_f$ is called the horizontal lift of $X \in T_{\pr_M(f)}M$ at $f$, if $(\pr_{M})_* v = X$. Since $(\pr_{M})_*|_{E_f}$ is a vector space isomorphism, the mapping $X\mapsto v$ is well defined. Write $H_X(f)$ to denote the horizontal lift of $X$ at $f$. If $x\colon [0,\tau] \to M$ is any absolutely continuous curves in $M$, and $f\colon[0,\tau]\to{\mathcal F}(M)$ is defined so that each $f_j(t)$ is a parallel vector field along $x(t)$, then
\begin{equation} \label{liftder} H_{\dot{x}(t)}(f(t)) = \dot{f}(t).\end{equation}
The horizontal curve $f(t)$ solving~\eqref{liftder} is completely determined up to initial configuration $f(0)\in{\mathcal F}_{x(0)}(M)$.

\begin{definition}
A curve $y\colon [0,\tau] \to \real^n, y(0)= 0,$ is called the anti-development of $x\colon [0,\tau]\to M$, if there is a horizontal curve $f\colon [0,\tau]\to{\mathcal F}(M)$, so that $\pr_{M} f(t) = x(t),$ and
\begin{equation} \label{antidev}
f(t)(\dot{y}(t)) = \dot{x}(t).\end{equation}
\end{definition}
Note that~\eqref{liftder} permits to rewrite equation~\eqref{antidev} as $\dot{f}(t) = H_{f(t)(\dot{y}(t))}(f(t)),$ which corresponds to the definition of anti-development often found in the literature.

 If $\nabla$ is compatible with the metric then, according to the definition of $F(M)$ from Subsection~\ref{subsec:recalling}, we can consider the corresponding Ehresmann connection $E$ as a subbundle of $TF(M)$, since orthonormal frames remain orthonormal under parallel transport. 
 
The idea of development --a sort of ``reverse'' of the anti-development-- plays a fundamental role when defining Brownian motion on Riemannian manifolds. In~\cite[Chapter 2]{Hsu} it is possible to find the comment that the development corresponds to a rolling with no slipping of $M$ on ${\mathbb R}^n$, but no further interpretation is given.

\begin{remark} \label{remarkantidev}
For a horizontal curve $f\colon [0,\tau]\to{\mathcal F}(M)$ such that $\pr_M f(t) = x(t)$, the corresponding anti-development is given explicitly by $y(t) = \int_0^t f^{-1}(s)(\dot{x}(s)) ds.$
\end{remark}

For the rest of this paper, whenever we refer to a horizontal curve $f(t)$ in $F(M)$, we mean that $f(t)$ is horizontal with respect to the Ehresmann connection associated to the Levi-Civita connection on $M$. Consequently, any anti-development curve is defined with respect to such horizontal curves.

\subsection{Anti-development and intrinsic rolling} \label{devandroll}

The following result connecting horizontal curves in the frame bundles with rollings can be found in \cite[Corollary~1]{Grong}.

\begin{lemma} \label{lemmaantidevroll}
For any rolling $q(t)$ of $M$ on $\widehat M$, there are horizontal curves $f\colon [0,\tau]\to F(M)$ and $\widehat f\colon [0,\tau]\to F(\widehat M)$, so that $\pr_{M} f(t) = \pi q(t)$, $\pr_{\widehat M} \widehat f(t) = \widehat \pi q(t)$ and $q(t) = \widehat f(t) \circ f^{-1}(t).$
\end{lemma}

In particular, any rolling $q(t)$ of $\real^n$ on $M$ can be considered as a horizontal curve in $F(M)$. If $y(t)$ is the projection of $q(t)$ to $\real^n$ and $x(t)$ is the corresponding curve in $M$, then, up to translation, $y(t)$ is an anti-development curve of $x(t)$. In this context, equation~\eqref{antidev} becomes a restatement of the no slip condition, while the requirement of $f(t)$ being horizontal is equivalent to the no twist condition. The latter follows since horizontality is equivalent to require that $f(t)$ sends parallel vector fields along $y(t)$ (i.e. constant vector fields) to parallel vector fields along~$x(t)$.

From the definition of rolling, the following transitivity holds (cf. \cite[Appendix~B, Theorem 4.1]{Sharpe}): if $q(t)$ is a rolling along the curves $x(t)$ in $M$ and $\widehat x(t)$ in $\widehat M$, and if $\widetilde q(t)$ is a rolling along the curves $\widehat x(t)$ in $\widehat M$ and $\widetilde x(t)$ in $\widetilde M$, then $\widetilde q(t) \circ q(t)$ is a rolling along $x(t)$ in $M$ and $\widetilde x(t)$ in $\widetilde M$. This transitivity, together with Lemma~\ref{lemmaantidevroll}, imply the following result.

\begin{proposition} \label{PropAntidev}
Consider two curves $x\colon [0,\tau]\to M$ and $\widehat x\colon [0,\tau]\to \widehat M$, and let $y(t)$ and $\widehat y(t)$ be anti-development curves for $x(t)$ and $\widehat x(t)$ respectively. Then there exist a rolling of $M$ on $\widehat M$ along $x(t)$ and $\widehat x(t)$ if and only if there is a rolling of $\real^n$ on itself along $y(t)$ and $\widehat y(t)$.
\end{proposition}

This result will be used later in Section \ref{sec:AbsConCurves} in order to reduce the question of existence of rollings between manifolds to rollings along curves in $\real^n$.

\subsection{Rolling along a loop}\label{sec:dim2}

Using the connection between rollings and the horizontal curves in the frame bundles, we state some corollaries of Theorem~\ref{geodcurv}. In particular, we want use the previous mentioned theorem to solve the following problem. Consider a rolling $q(t)$ of $M$ on $\real^2$ along a given curve $x(t)$ in $M$. Assume that $x(t)$ is a continuous loop, i.e. $x(0) = x(\tau)$. When will $q(t)$ be a continuous loop also?

On the way to solving this, let us first look at a rolling $q(t)$ of surfaces $M$ and $\widehat M$ such that $x(t) = \pi q(t)$ and $\widehat x(t) = \widehat \pi q(t)$ are loops. We want to argue that we can determine whether or not $q(t)$ is a loop itself from the projected curves, at least if they are sufficiently regular.

\begin{corollary} \label{corloop1}
 Let $M$ and $\widehat M$ be surfaces. Let $q\colon [0,\tau]\to Q$ be a rolling such that $x(t) = \pi q(t)$ and $\widehat x(t) = \widehat \pi  q(t)$ are $C^2$ curves parametrized by arc-length. Then $q(t)$ is a continuous loop in $Q$ if and only if  both $x(t)$ and $\widehat x(t)$ are continuous loops where the oriented angles $\angle(\dot x(0),\dot x(\tau))$ and $\angle(\dot {\widehat x}(0),\dot {\widehat x}(\tau))$ coincide.
\end{corollary}

\begin{proof}
Consider the orthonormal positively oriented frame $f(t)=(f_1(t), f_2(t))$, parallel along $x(t)$ with $\dot x(0) = f_1(0)$. Let $\widehat f(t)=(\widehat f_1(t),\widehat f_2(t))$, where $\widehat f_j(t) = q(t) f_j(t), j=1,2$. Consider a $C^1$ curve $v(t)$ in $\real$, so that
\begin{equation} \label{parallelparm}
\begin{array}{rcl}
\dot x(t) & = & \cos(v(t)) f_1(t) + \sin(v(t)) f_2(t), \\
\dot {\widehat x}(t) & = & \cos(v(t))\widehat f_1(t) + \sin(v(t))\widehat f_2(t). 
\end{array}
\end{equation}
It is easy to see that the oriented curvature of $x(t)$ satisfies $k_g(t) = \dot v(t).$ Define $\int_0^\tau k_g(t)\,dt = v(\tau) - v(0) =: \alpha$.

Let $\theta$ be the angle in which $f(\tau)$ is rotated with respect to $f(0)$. Define the angle $\widehat \theta$ similarly. By definition, $q(0) = q(\tau)$ if and only if $\theta = \widehat \theta$ holds. Since $\angle(\dot x(0),\dot x(\tau))=\alpha + \theta$ and $\angle(\dot {\widehat x}(0),\dot {\widehat x}(\tau))=\alpha + \widehat \theta,$ the result follows.
\end{proof}

\begin{center}
\begin{figure}[h]
\scalebox{1} % Change this value to rescale the drawing.
{
\begin{pspicture}(0,-1.95)(8.32,1.97)
\pscircle[linewidth=0.04,dimen=outer](4.7,0.47){1.2}
\psbezier[linewidth=0.02](3.5,0.47)(3.5,-0.23)(5.9,-0.23)(5.9,0.47)
\psbezier[linewidth=0.02,linestyle=dashed,dash=0.16cm 0.16cm](3.5,0.47)(3.5,1.17)(5.9,1.17)(5.9,0.47)
\psdots[dotsize=0.12](4.7,-0.73)
\psline[linewidth=0.04cm](2.4,0.17)(0.9,-1.93)
\psline[linewidth=0.04cm](0.9,-1.93)(6.8,-1.93)
\psline[linewidth=0.04cm](8.3,0.17)(6.8,-1.93)
\psline[linewidth=0.04cm](2.4,0.17)(3.6,0.17)
\psline[linewidth=0.04cm](5.8,0.17)(8.3,0.17)
\psline[linewidth=0.04cm,linestyle=dotted,dotsep=0.16cm](3.6,0.17)(5.8,0.17)
\usefont{T1}{ptm}{m}{n}
\rput(2.08,-1.625){$\mathbb{R}^2\cong T_pS$}
\usefont{T1}{ptm}{m}{n}
\rput(4.62,-1.125){$\alpha$}
\usefont{T1}{ptm}{m}{n}
\rput(5.38,1.775){$S$}
\psline[linewidth=0.04cm,arrowsize=0.05291667cm 2.0,arrowlength=1.4,arrowinset=0.4]{->}(4.7,-0.73)(5.4,-1.13)
\psline[linewidth=0.04cm,arrowsize=0.05291667cm 2.0,arrowlength=1.4,arrowinset=0.4]{->}(4.7,-0.73)(3.8,-0.93)
\psbezier[linewidth=0.04](4.38,-0.81)(4.6,-0.93)(4.7,-0.93)(4.9,-0.83)
\usefont{T1}{ptm}{m}{n}
\rput(4.74,0.575){$x(t)$}
\usefont{T1}{ptm}{m}{n}
\rput(3.61,-0.725){$\dot x(0)$}
\usefont{T1}{ptm}{m}{n}
\rput(5.49,-0.825){$\dot x(\tau)$}
\psbezier[linewidth=0.04](4.7,-0.748404)(5.2,-0.63)(5.1,-0.447208)(5.4,0.07)(5.7,0.587208)(5.5903344,1.0087012)(4.6,0.87)(3.6096656,0.7312988)(4.4,0.66594887)(4.3,0.37)(4.2,0.07405113)(3.5662777,0.74515265)(3.9,0.07)(4.233722,-0.60515267)(4.4,-0.63)(4.7,-0.748404)
\usefont{T1}{ptm}{m}{n}
\rput(4.67,-0.525){$p$}
\psline[linewidth=0.02cm](5.56,0.53)(5.44,0.37)
\psline[linewidth=0.02cm](5.56,0.53)(5.594142,0.33100504)
\end{pspicture} 
}
\caption{A sphere $S$ rolling following a loop $\widehat x(t)$ in ${\mathbb R}^{2}$.}\label{fig:loop}
\end{figure}
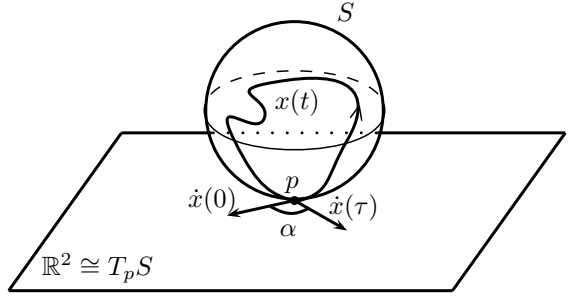
\end{center}

The angles $\theta$ and $\widehat \theta$ in the proof of Corollary~\ref{corloop1}, can be understood as elements in the respective holonomy groups corresponding to $x(t)$ and $\widehat x(t)$. Hence, another way to formulate Corollary~\ref{corloop1}, is to say that $q(0) = q(\tau)$ if and only if the loops $x(t)$ and $\widehat x(t)$ have the same holonomy. Note that, since we are working in dimension 2, the statement ``have the same holonomy'' actually makes sense, since $\SO(T_xM, T_xM)$ is canonically isomorphic to $\SO(T_{\widehat x}\widehat M, T_{\widehat x} \widehat M)$ for any $x \in M, \widehat x \in \widehat M$.

In particular, if we are rolling $M$ on $\real^2$ along a loop $x(t)$ in $M$, then holonomy of $x(t)$ can be identified with the angle that $M$ is rotated with after the rolling is complete relative to the standard basis in $\real^2$. This statement holds even when the curve $y(t)$ obtained in $\real^2$ by the rolling, is not a loop. Hence, requiring only trivial holonomy of $x(t)$ is not enough to solve the question of when $q(t)$ will be a loop. However, by again using Theorem~\ref{geodcurv}, we are able to find a solution.

\begin{corollary}\label{intkappa}
Let $x(t)$ be a $C^2$ curve parametrized by arc-length in a surface $M$. Then the following are equivalent.
\begin{itemize}
\item[(a)] Any rolling $q(t)$ of $M$ on $\real^2$ along $x(t)$ has the same initial and final configuration.
\item[(b)] $x(t)$ is a continuous loop with trivial holonomy and with geodesic curvature $k_g(t)$ satisfying 
\begin{equation} \label{eq:eIntegral0} \int_0^\tau \exp\left(i\int_0^t k_g(s) ds \right) dt = 0.\end{equation}
where $i$ is the imaginary unit.
\end{itemize}
Furthermore, $q(t)$ is a $C^1$ loop if and only if $x(t)$ satisfies (b) and $\int_0^\tau k_g(t) dt = 0 \, (\text{mod } 2\pi).$
\end{corollary}

\begin{proof}
Let $q(t)$ be an arbitrary rolling of $M$ on $\real^2$ along $x(t)$. Let $y(t)$ be the projection of $q(t)$ to $\real^2$. Let $\{e_1, e_2 \}$ be a rotation of the standard basis of vector fields on $\real^2$ such that $\dot y(0) = e_1(y(0))$, and consider a basis of parallel vector fields along $x(t)$ given by $\{ f_1(t), f_2(t) \} = \{ q(t)^{-1} e_1(y(t)), q(t)^{-1} e_2(y(t)) \}.$ Similarly to \eqref{parallelparm}, if we can write $v(t) = \int_0^t k_g(s) ds,$
and
$$\dot x(t) = \cos(v(t)) f_1(t) + \sin(v(t)) f_2(t), \quad \dot y(t) = \cos(v(t)) e_1(y(t)) + \sin(v(t)) e_2(y(t)).$$
Then it is clear that $y(t)$ is a loop if \eqref{eq:eIntegral0} holds and we see that both $\dot x(0)= \dot x(\tau)$ and $\dot y(0) = \dot y(\tau)$ holds if also $v(0) = v(\tau) \, (\text{mod } 2 \pi)$.

The only thing that remains to be proven, is that $q(t)$ is a $C^1$ loop when both $x(t)$ and $y(t)$ are $C^1$ loops. But this simply follows from the fact that $q(t)$ is constant in the basis $\{f_1, f_2\}$ and $\{e_1, e_2\}$ and the mappings $t \mapsto f_j(t)$ and $t \mapsto e_j(y(t))$ are obviously $C^1$ loops.
\end{proof}

\section{Existence of intrinsic rollings in dimension $n$}\label{sec:rolldimn}

We will now turn to the general case of $n$ dimensional manifolds. We provide results for when a rolling exists along a given pair of curves. At the same time, we want to show that rolling preserves local data, and when curves are not nice enough to have well defined geodesic curvatures, we will give an argument for why the anti-development can be considered as ``local data''. To build up intuition for this argument, we will start with the simplest case of geodesic, and then gradually look at more complicated curves.

\subsection{Rolling along a geodesics}
Assume that $q(t)$ is a rolling of $M$ on $\widehat M$, with initial configuration $q(0) = q_0\colon T_{x_0}M \to T_{\widehat x_0} \widehat M$. Assume that $\pi q(t) =  x(t)$ is a geodesic. Then $\widehat \pi q(t) = \widehat x(t)$ is also a geodesic, since $\frac{D}{dt} \dot{\widehat x}(t) = q(t) \frac{D}{dt} \dot x(t) = 0.$ Let $y(t)$ and $\widehat y(t)$ be respective anti-development curves of $x(t)$ and $\widehat x(t)$ corresponding to horizontal curves $f(t)$ and $\widehat f(t)$ in the respective frame bundles. These will both be straight lines since $\dot x(t)$ and $\dot{\widehat x}(t)$ are parallel vector fields. Write them as $y(t) = ty_0$ and $\widehat y(t) = t \widehat y_0$. By the definition of the Riemannian exponential and \eqref{antidev}, we know that,
$$x(t) = \exp_{x_0}\left(t f(0)(y_0)\right), \qquad \widehat x(t) = \exp_{\widehat x_0}\left(t \widehat f(0)(\widehat y_0)\right).$$
By the no slipping condition, we know that $q_0(f(0)(y_0)) = \widehat f(0)(\widehat y_0)$. Hence, for small values of $t$, $\widehat x(t)$ can be written as $\widehat x(t) = \exp_{\widehat x_0} \circ q_0 \circ \exp_{x_0}^{-1}(x(t)).$

\subsection{Curves with $C^{n-1}$-regularity}

In order to state an analogue of Theorem~\ref{geodcurv} for dimensions higher than 2, we need to find a suitable definition of geodesic curvature. The notion we use can be found in~\cite[Chapter 7.B]{Spivak4}, and is formulated in the following way.

Let $M$ be a oriented Riemannian manifold of dimension $n$ and let $x:[0,\tau]\to M$ be a curve of class $C^{n+1}$ parameterized by arc-length. Define the unit vector field $v_1(t) = \dot{x}(t)$ along $x(t)$, and let $\kappa_1(t) = \left|\frac{D}{dt}\, v_{1}(t) \right|$. Assuming $\kappa_1(t)$ never vanishes, there is a unique unit vector field $v_2(t)$ along $x(t)$ satisfying
$\frac{D}{dt}\, v_{1}(t) = \kappa_1(t) v_2(t)$.
Inductively, assume that $\kappa_{i}(t)$ and $v_{i+1}(t)$ are well-defined for $i < j$, where $j<n$ is fixed. Denote
\[\kappa_j(t) = \left|\dfrac{D}{dt}\, v_j + \kappa_{j-1}(t) v_{j-1}\right|.\]
If $\kappa_{j}(t)$ never vanishes, define $v_{j+1}$ to be the unit vector field along $x(t)$ satisfying
\begin{equation}\label{eq:defkappa}
\frac{D}{dt}\, v_{j}(t) + \kappa_{j-1}(t) v_{j-1}(t) = \kappa_{j}(t) v_{j+1}(t).
\end{equation}

Simple calculations show that $\langle v_{i}(t),v_{j}(t)\rangle=\delta_{i,j}$ for all $i,j$.

\begin{definition}
The unit vector field $v_j(t)$ in~\eqref{eq:defkappa} is called the the $j-$th Frenet vector field of $x(t)$. The function $\kappa_j(t)$ is called the $j-$th geodesic curvature of $x$.
\end{definition}

A $C^n$ curve $x(t)$ in $M$, is called $C^k$-regular, where $1 \leq k \leq n$, if
$$\left\{\dot x(t),\frac{D}{dt} \dot x(t), \frac{D^2}{dt^2} \dot x(t), \dots, \frac{D^k}{dt^k} \dot x(t) \right\},$$
are linearly independent for every $t$. The $k-$th Frenet vector field $v_k(t)$ exists if and only if $x(t)$ is $C^k$-regular. We want an alternative way to define $\kappa_{n-1}$ and $v_n$ which only requires the curve to be $C^n$ and have $C^{n-1}$-regularity and which also encodes the orientation into the curvatures. This approach can also be found in~\cite[Chapter 7.B]{Spivak4}.

Assume that $x(t)$ is a $C^n$ curve, which is $C^{n-1}$-regular, and consider the curvatures $\kappa_1, \dots, \kappa_{n-2}$ and the Frenet vector fields $v_1, \dots, v_{n-1}$ defined above. Then $v_n$ is defined as the unique unit vector field so that $v_1(t), \dots, v_n(t)$ is a positively oriented orthonormal basis for every $t$. Formally $v_n = \sharp (\star(\flat(v_1 \wedge \cdots \wedge v_{n-1})))$, where $\flat$ and $\sharp$ are the musical isomorpisms and $\star$ is the Hodge star operator corresponding to the metric on $M$. The geodesic curvature $\kappa_{n-1}$ is subsequently defined as
$$\kappa_{n-1}(t) = \left\langle v_n(t), \frac{D}{dt}\, v_{n-1}(t) + \kappa_{n-2}(t) v_{n-2}(t) \right\rangle.$$
This curvature can have both positive and negative values, and will change sign if we change orientation on $M$. Notice also that with this definition, $v_n$ is at least a $C^1$ vector field, while $\kappa_{n-1}$ is in general only continuous.

\begin{remark}
If a curve is $C^k$-regular for $k \geq 1$, this implies that the derivative of $x(t)$ never vanishes, which again implies that we are able to reparametrize the curve by arc-length without loosing differentiability at any point. Hence, instead of requiring that the curves are parametrized by arc-length, we could have just required $C^1$-regularity.
\end{remark}

\begin{lemma} \label{lemma:Cn}
Assume that $x(t)$ is a $C^n$ curve in $M$ and let $q(t)$ be a rolling of $M$ on $\widehat M$ along $x(t)$. Then the curve $q(t)$ and $\widehat x(t) = \widehat \pi q(t)$ are also $C^n$. Furthermore, if $x(t)$ is $C^k$-regular, $1 \leq k \leq n$, then so is $\widehat x(t)$.
\end{lemma}

\begin{proof}
To see that $q(t)$ and $\widehat x(t)$ are $C^n$, it is sufficient to show this around a point $q_0: T_{x_0}M \to T_{\widehat x_0} \widehat M$ in the image of $q(t)$. Pick local coordinates, $(\xi, U)$ and $(\widehat \xi, \widehat U)$ around $x_0$ and $\widehat x_0$, respectively, and choose orthonormal bases of vector fields $e_1, \dots, e_n$ on $U$ and $\widehat e_1, \dots, \widehat e_n$ on $\widehat U$. For simplicity, we will write vector fields $\frac{\partial}{\partial{\xi_i}}$ simply as $\partial_{\xi_i}$, and we will use similar conventions on $\widehat M$.

Define matrix-valued functions $\phi = (\phi_{ij})$ on $U$, and similarly $\widehat \phi = (\widehat \phi_{ij})$ on $\widehat U$, by
$$e_j = \sum_{i=1}^n \phi_{ij} \partial_{\xi_i} , \qquad \widehat e_j = \sum_{i=1}^n \widehat \phi_{ij} \partial_{\widehat \xi_i}.$$
We write $\dot x(t) =: \sum_{i =1}^n \dot \xi_j(t) \partial_{\xi_i}|_{x(t)},$ and similarly $ \dot{ \widehat x}(t) =: \sum_{i =1} \dot{\widehat \xi}_i(t) \partial_{\widehat \xi_i}|_{\widehat x(t)}$.
Denote by $[q]=(q_{ij})\in\SO(n)$ the matrix 
%of $q$ in the bases $e_1, \dots, e_n$ and $\widehat e_1, \dots, \widehat e_n$, i.e. $[q] = (q_ij)$, where
with entries $$q_{ij}(t)\colon = \langle \widehat e_i(\widehat x(t)), q(t) e_i(x(t)) \rangle$$ and the corresponding Christoffel symbols
%are defined with respect to our orthonormal bases
$$\Gamma_{kj}^i(x) = \langle e_i(x), \nabla_{e_k(x)} e_j(x) \rangle, \quad \widehat \Gamma_{kj}^i(\widehat x) = \langle \widehat e_i(\widehat x), \nabla_{\widehat e_k(\widehat x)} \widehat e_j(\widehat x)\rangle.$$
Consider the antisymmetric matrices $\Gamma_k := (\Gamma_{kj}^i)$, and $\Gamma_k := (\Gamma_{kj}^i)$. Then, from \cite{GGMS}, we need to solve the equations
\begin{eqnarray} 
\label{solve1} \dot{ \widehat \xi}_i &=& \sum_{j,k,l =1} \widehat \phi_{il} q_{lk} \phi^{kj} \dot \xi_j,\\
 \label{solve2} \dot q_{ij} &=& \sum_{k,l,r =1}^n \phi^{lk} \dot \xi_k \left( q_{ir}\Gamma_{lj}^r - \sum_{s=1}^n q_{rj} q_{sl} \widehat \Gamma_{sr}^i \right),
 \end{eqnarray}
where, as usual, $\phi^{-1}=(\phi^{ij})$. In matrix form, they read
\begin{eqnarray*}
\dot{\widehat \xi} &=& \widehat \phi [q] \phi^{-1} \dot\xi,\\
\dot{[q]} &=& \sum_{k,l =1}^n \phi^{lk} \dot \xi_k \left([q] \Gamma_{l}
- \sum_{s=1}^n q_{sl} \widehat \Gamma_s [q] \right).
\end{eqnarray*}
Obviously, it follows that these equations have a $C^n$ solution if $\xi(t)$ is $C^n$.

If $x(t)$, is $C^k$-regular, then from the fact that $q(t)$ is an invertible linear map for each $t$, the vectors
$$\left\{ q(t) \dot x(t), q(t)\frac{D}{dt} \dot x(t), \dots, q(t) \frac{D^k}{dt^k} \dot x(t) \right\} = \left\{ \dot{\widehat x}(t), \frac{D}{dt} \dot{\widehat x}(t), \dots, \frac{D^k}{dt^k} \dot{\widehat x}(t) \right\},$$
must also be linearly independent.
\end{proof}

With all of this terminology set up, we are able to present a generalization of Theorem~\ref{geodcurv}.

\begin{theorem} \label{curvaturetheorem1}
Let $x\colon [0,\tau]\to M$ and $\widehat x\colon [0,\tau]\to\widehat M$ be two $C^n$ curves, parametrized by arc-length, that are also $C^{n-1}$-regular. Then, there is a rolling $q(t)$ such that
$$\pi q(t) = x(t), \qquad \widehat \pi q(t) = \widehat x(t)$$
if and only if $\kappa_j = \widehat{\kappa}_j,$ for $j=1,\ldots, n,$ where $\kappa_j$ and $\widehat \kappa_j$ are the respective geodesic curvatures of $x(t)$ and $\widehat x(t)$. 
\end{theorem}

\begin{proof}
Write $\{v_j\}_{j=1}^n$ and $\{ \widehat v_j\}_{j=1}^n$ for the Frenet vector fields along $x$ and $\widehat x$.

Assume first that there is a rolling $q(t)$ along $x(t)$ and $\widehat x(t)$. From the no slip condition, we know that $q(t) v_1 (t) = \widehat v_1(t)$. From the non twist property and induction, we have $q(t) \kappa_{j-1}(t) v_j = \widehat \kappa_{j-1}(t) \widehat v_j(t)$. Hence $\kappa_{j-1}(t)=\widehat \kappa_{j-1}(t)$.

Conversely, assume that $\kappa_{j-1}(t)=\widehat \kappa_{j-1}(t)$. Define $q(t)$ by the formula $\widehat{v}_j(t)=q(t) v_j(t)$.
In order to see that $q(t)$ is a rolling, we need to show that if $w(t)=\sum_{j=1}^nw_j(t)v_j(t)$ is any vector field along $x(t)$, we have $\frac{D}{Dt} q(t) w(t) = q(t) \frac{D}{dt} w(t).$
This equality holds since (we introduce the notation $\kappa_0 = 0$ to simplify formulas)
\begin{multline*}
q(t) \frac{D}{dt}\, w(t)  = q(t)\sum_{j=1}^n \frac{D}{dt}\, w_j(t) v_j(t) =\\
= \sum_{j=1}^n \Big(\dot{w}_j(t) \widehat{v}_j(t) + w_j(t) \big(-\kappa_{j-1}(t) \widehat{v}_{j-1}(t) + \kappa_{j}(t) \widehat{v}_{j+1}(t)\big) \Big)\\
=\sum_{j=1}^n\Big(\dot w_j(t)\widehat v_{j}(t)+w_j(t)\frac{D}{dt}\,\widehat v_j(t)\Big)= \frac{D}{dt}\, q(t) w(t).
\end{multline*}
This concludes the proof.
\end{proof}

\subsection{Condition for general curves} \label{sec:AbsConCurves}

Not all curves in higher dimensions have well defined Frenet vector fields or curvatures. Even the failure of this to exist in one point, can hinder the existence of a rolling.

\begin{example} \label{exonepoint}
Consider the curves in $\real^3$
$$y(t) = \left\{\begin{array}{ll} \big(t, e^{-\frac{1}{t^2}}, 0 \big) & \text{if } t \neq 0 \\
\big(0,0,0\big) & \text{if } t = 0 \end{array} \right. ,
\quad \widehat y(t) = \left\{\begin{array}{ll} \big(t, e^{-\frac{1}{t^2}}, 0\big) & \text{if } t < 0 \\
\big(0,0,0\big) & \text{if } t = 0 \\
\big(t, 0, e^{-\frac{1}{t^2}}\big) & \text{if } t > 0 \end{array} \right. .$$
Both curves have coinciding curvatures when restricted to either $t >0$ or $t<0$. However, we cannot construct a rolling of $\real^3$ on itself along $y(t)$ and $\widehat y(t)$, by the following result.
\end{example}

\begin{proposition}
Let $q(t): [0,\tau] \to Q$ be a rolling of $M$ on $\widehat M$. Let $0 \leq a < b < c \leq \tau$, and assume that $x(t) = \pi q(t)$ is $C^n$ and $C^{n-1}$-regular on $(a,b)$ and $(b,c)$. Let $v_1, \dots, v_n$ and $w_1, \dots, w_n$ be the Frenet frames of $x(t)$ on respectively $(a,b)$ and $(b,c)$. Extend both frames to the point $b$ by continuity. Let $\widehat v_1, \dots, \widehat v_n$ and $\widehat w_1, \dots, \widehat w_n$ be defined similarly for $\widehat x = \widehat \pi q(t)$. Then for any $i,j$,
$$\langle \widehat v_i(b), \widehat w_j(b) \rangle = \langle v_i(b), w_j(b) \rangle.$$ 
\end{proposition}

\begin{proof} This is a simple consequence of the equalities $q(t) v_i(t) = \widehat v_i(t)$ and $q(t) w_j(t) = \widehat w_j(t)$. \end{proof}

Another way of stating this, is that if there is an isolated point where $C^{n-1}$-regularity fails, permitting a possible rotation in the Frenet frame, then this rotation should be the same for both $x(t)$ and $\widehat x(t)$. From this, it seems that rolling without twisting or slipping preserves the local structure very well. We want to show that a rolling has this property for any absolutely continuous curve. However, it seems unclear what ``local structure'' means for a curve which does not have a Frenet frame, nor geodesic curvatures.

Our idea is that anti-development curves can be seen as a generalization of the geodesic curvatures. Note that we can define a curve $\kappa(t) = (\kappa_1(t), \kappa_2(t), \dots, \kappa_{n-1}(t))$ in $\real^{n-1}$ describing the local structure of a $C^{n-1}$-regular curve $x(t)$ in $M$, parame\-trized by arc-length. If we drop the last requirement, we also need an $n$-th coordinate in form of the speed $s(t) = |\dot x(t)|$ in addition to the curvatures. Given a starting point and an initial value for the corresponding Frenet frame, the curvatures and $s(t)$ determine a curve uniquely, which always exists for short time, and for all time if $M$ is complete. Furthermore, if $M$ is also connected and simply connected with a constant sectional curvature, in addition to being complete, then the data $(\kappa_1(t), \dots, \kappa_{n-1}(t))$ and $s(t)$ determines $x(t)$ uniquely up to an isometry (see \cite[Corollary 4]{Spivak4}).

Similarly, an anti-development $y(t)$ of a curve $x(t)$ consists of $n$ coordinates. Given a starting point and an initial value for the corresponding horizontal curve in the frame bundle, $y(t)$ determines a curve uniquely, which always exists for short time, and for all time if $M$ is complete. If $x(t)$ happens to be $C^{n-1}$-regular, then the geodesic curvatures of $x(t)$ and the speed is encoded into $y(t)$, since the anti-development will have the same speed and geodesic curvatures. We want to show that for a pair of general curves, their anti-development curves determine if there exists a rolling along these curves, in the same way the curvatures did for $C^{n-1}$-regular curves.

In Subsection \ref{sec:constantK}, we will complete the analogy by showing that in a complete, connected, simply connected manifold $M$ of constant curvature, any curve is uniquely determined up to isometry by its anti-development.

\begin{theorem} \label{generalcurve}
Let $x\colon[0,\tau]\to M$ and $\widehat x\colon[0,\tau]\to \widehat M$ be absolutely continuous curves. Let $y(t)$ and $\widehat y(t)$ be any anti-development curves of $x(t)$ and $\widehat x(t)$, respectively. Then there is a rolling $q(t)$, with
$$\pi q(t) = x(t), \qquad \widehat \pi q(t) = \widehat x(t),$$
if and only if $\widehat y(t) = \iota(y(t))$ for some $\iota \in \SO(n)$. In other words, there exist a rolling along a pair of curves $x(t)$ and $\widehat x(t)$ if and only if they have the same set of anti-development curves.
\end{theorem}

%\begin{lemma} \label{lemmaGGMS}
%Let $q\colon[0,\tau]\to Q$ be a rolling of $M$ on $\widehat M$, and let $\pr_{M \times \widehat M}\circ q(t) = (x(t), \widehat x(t))$.
%Let $\{f_j(t)\}_{j=1}^n$ be a frame of positively orthonormal parallel vector fields along $x(t)$.
%Let $\{\widehat f_j(t)\}_{j=1}^n$ be defined similarly along $\widehat x(t)$.
%Then the matrix $(q_{ij}) = \big(\langle \widehat f_i(t), q(t) f_j(t) \rangle\big)$ is a constant matrix in $\SO(n)$.
%\end{lemma}
The proof of this theorem follows by combining Proposition \ref{PropAntidev} with the following lemma.

\begin{lemma} \label{generalcondition} 
Let $y,\widehat y\colon[0,\tau]\to\real^n$ be two absolutely continuous curves such that $y(0)= \widehat y(0) = 0$. Then there exists a rolling of $\real^n$ on itself, along $y(t)$ and $\widehat y(t)$ if and only if there is an $\iota \in \SO(n)$ so that $\iota(y(t)) =\widehat y(t).$
\end{lemma}

\begin{proof}
Assume that a rolling exists. Let $r = (r_1, \dots, r_n)$ be the canonical coordinates in $\real^n$, and let $e_j := \partial_{r_j}$. A general property of a rolling found in \cite[Lemma 1]{GGMS} states that if $q(t)$ is a rolling of $M$ on $\widehat M$ along $x(t)$ and $\widehat x(t)$, and if $f_1, \dots, f_n$ and $\widehat f_1, \dots, \widehat f_n$ are orthonormal bases of parallel vector fields along $x(t)$ and $\widehat x(t)$, then the matrix $(q_{ij}(t))$, given by $$q_{ij}(t) =  \langle q(t) f_j(t), \widehat f_i(t) \rangle$$ is a constant matrix in $\SO(n)$. In our case, this means that
$$(q_{ij}) = \big(\langle e_i(\widehat y(t)), q(t) e_i(y(t)) \rangle\big)$$
is constant in $\SO(n)$.

From the no slip condition
$$\dot{\widehat y} = \sum_{i = 1}^n \dot{\widehat y}_i e_i = \sum_{i,j=1}^n q_{ij} \dot y_j e_i.$$
Solving this, we obtain $\widehat y(t) = (\widehat y_1(t), \dots, \widehat y_n(t))$, where $\widehat y_j(t) = \sum_{j=1} q_{ij} y_j(t).$
Hence, by setting $\iota=(q_{ij})$ in the standard basis, then $\widehat y(t) = \iota( y(t)).$

The converse follows by taking $q(t) = \iota_{*,y(t)}$.
\end{proof}

Theorem \ref{generalcurve} may be seen as a generalization of Theorem \ref{curvaturetheorem1} since, if we have a curve in $\real^n$ starting at 0 whose Frenet vector fields all exist, then it is determined, up to a linear isometry, by its curvatures.

\subsection{Distribution along curves}

As seen in Example \ref{exonepoint}, if a geodesic curvature (not the top one) of a curve vanishes somewhere, then it is difficult to determine the existence of a rolling.
% from geometric information given by the curvature.
However, if a certain number of geodesic curvatures vanish identically, we can obtain some results.

\begin{definition}
Let $x\colon [0,\tau]\to M$ be an absolutely continuous curve. For each $s\in(0,\tau)$, associate a $j$-dimensional subspace $V(s) \subset T_{x(s)}M$, so that $V$ forms a distribution along $x$. Then $V$ is called parallel along $x$, if it is closed under parallel transport.
\end{definition}

An equivalent characterization for parallel distributions is that they are closed under covariant derivative, as observed in~\cite[Prelemma 7]{Spivak4}.

\begin{lemma} \label{subspacelemma}
Let $x\colon [0,\tau]\to M$ be an absolutely continuous curve. Then there is a parallel distribution $V$ along $x(t)$ of rank $k$ containing $\dot x(t)$ if and only if there is a lifting to a curve $f(t)$ in $F(M)$ of parallel vector fields along $x(t)$, such that the corresponding anti-development
\begin{equation} \label{planecurve}
y(t) = \int_0^t f^{-1}(s)(\dot x(s)) ds,
\end{equation}
is a curve in $\real^{k} \times \{0\} \subseteq \real^n.$
\end{lemma}

\begin{proof}
First assume there is a distribution $V$ parallel along $x(t)$, with $\rank V = k$ and containing $\dot x(t)$. Let $f_1,\dots, f_k$ be a basis of parallel vector fields of $V$, and let $f_{k+1}, \dots, f_n$ be a basis of parallel vector fields of $V^\perp$. Since $\dot x(t)$ is in $V$, from equation~\eqref{planecurve}, it is clear that $y(t)$ is contained in $\real^k \times \{ 0 \}$.

Conversely, assume that $y(t)$ is in $\real^k\times \{0\}$. Since $f(t)(\dot y(t))  = \dot x(t)$, we know that if $V$ is the distribution along $x(t)$ spanned by $f_1, \dots, f_k$, it contains $\dot x(t)$.
\end{proof}

With this lemma at hand, we can show the existence of rollings when the curves have less regularity than in Theorem~\ref{curvaturetheorem1}.

\begin{proposition} \label{Prop0curvature}
Let $x\colon[0,\tau]\to M$ and $\widehat x\colon[0,\tau]\to \widehat M$ be $C^{k+1}$ curves parametrized by arc-length which are also $C^k$-regular. Let $V$ be the minimal parallel distribution along $x(t)$, such that $\dot x(t) \in V_{x(t)},$ and define $\widehat V$ similarly for $\widehat x(t)$.

Then if $\rank V = k$, there exist a rolling along $x(t)$ and $\widehat x(t)$ if and only if $\rank \widehat V = k$ and their $k-1$ first geodesic curvatures coincide.
\end{proposition}

\begin{proof}
From \cite[Theorem 1]{GGMS}, we know that $\rank \widehat V = k$. Furthermore, Lemma \ref{subspacelemma} tells us that there are curves $f\colon[0,\tau]\to F(M)$ and $\widehat f\colon[0,\tau]\to F(\widehat M)$ of parallel vector fields along $x(t)$ and $\widehat x(t)$, so that
$$y(t) = \int_0^t f^{-1}(s)(\dot x(s)) ds, \qquad \widehat y(t) = \int_0^t f^{-1}(s)(\dot{\widehat x}(s))ds,$$
are curves in $\real^k\times \{0\}$. We only need to show that there exists a rolling between $y(t)$ and $\widehat y(t)$, but this follows from Theorem \ref{curvaturetheorem1} and Proposition \ref{generalcurve}.
\end{proof}

\begin{remark}
If the manifolds in Proposition \ref{Prop0curvature} have constant sectional curvature, we can apply a result found in \cite[Lemma 8]{Spivak4}. Let $x(t)$ is a $C^1$-regular curve in $M$ with constant sectional curvature, and let $V(t)$ be a $k$-dimensional distribution along $x(t)$. Then the fact that $\dot x(t) \in V(t)$ for every $t$, implies that there exist a totally geodesic $k$-dimensional submanifold $N \subseteq M$, containing $x(t)$. Hence, if we have a rolling such as in Proposition \ref{Prop0curvature} of two manifolds with constant curvature, the system can be reduced to considering a rolling of two $k$-dimensional manifolds. \end{remark}

\subsection{Rolling and manifolds of constant curvature.} \label{sec:constantK}

For any $q_0 \in Q$, we define the orbit $\mathcal O_{q_0}$ of $q_0$ as the collection of all $q\in Q$, reachable from $q_0$ by a rolling. In~\cite[Proposition 3]{GGMS}, the problem of rolling was reformulated as the study of absolutely continuous curves almost everywhere tangent to a distribution $D$ over $Q$ of rank $n$. Curves in $Q$ tangent to $D$ are exactly the curves that satisfy \eqref{solve1} and \eqref{solve2}. We aim to exploit the case in which the manifolds rolling have constant sectional curvatures. To do that, we need the following lemma in~\cite[Corollary 5.23]{Chitour}.

\begin{lemma} \label{allndim}
The orbit $\mathcal O_{q_0}$ is an $n$ dimensional immersed manifold for any $q_0 \in Q$ if and only if $M$ and $\widehat M$ have constant and equal sectional curvature.
\end{lemma}

Write $\Isom(M,\widehat M)$ for the (possibly empty) collection of positively oriented isometries from $M$ to $\widehat M$.

\begin{theorem} \label{theoremisometries}
Let $M$ and $\widehat M$ be two connected Riemannian manifolds. For any $\iota \in \Isom(M,\widehat M),$ define the $n$-dimensional submanifold $\mathcal O_\iota = \{\iota_{*,x}\in Q \, \colon \, x \in M \}.$
Let $\mathscr O= \{ {\mathcal O}_\iota\, \colon \, \iota \in \Isom(M,\widehat M) \}$ be the collection of all such orbits. Then the following holds.

\begin{itemize}
\item[(a)] The mapping $\Phi\colon\Isom(M,\widehat M) \to \mathscr O$, $\iota\mapsto{\mathcal O}_\iota$,
is well defined and injective.

\item[(b)] If $M$ and $\widehat M$ are complete and simply connected, any $n$ dimensional orbit is of the form $\mathcal O_\iota$ for some $\iota\in \Isom(M, \widehat M)$.

\item[(c)] If $M$ and $\widehat M$ are isometric, complete, simply connected and of constant sectional curvature, then the map $\Phi$ is a bijection.
\end{itemize}
\end{theorem}

\begin{proof}
\begin{itemize}
\item[(a)] To see that $\mathcal{O}_\iota$ is an orbit of $D$, it is enough to show that $T_q {\mathcal{O}}_\iota = D_q$ for any $q \in{\mathcal{O}}_\iota.$ Since $\dim\mathcal O_\iota= \rank D$, we need to prove that for any absolutely continuous curve $x(t)$ in $M$, the map $\iota_{*,x(t)}$ is a rolling.

Clearly, $\iota_{*,x(t)}$ satisfies the no slipping condition. It also satisfies the non-twisting condition, since $t \mapsto \iota_{*,x(t)} Z(t)$ is a parallel vector field along $\iota \circ x(t)$ whenever $Z(t)$ is a parallel vector field along $x(t)$.

\item[(b)] Assume that $\mathcal{O}$ is an $n$ dimensional orbit of $D$. If $\widehat M$ is complete, then $\pi(\mathcal O) = M$.
To see this, assume that $q_0\in \mathcal O$ and $\pi(q_0)= x_0$. Then, for any $x_1\in M$, there is a rolling $q(t)$ from $x_0$ to $x_1$ such that $q(0) = q_0$, see \cite[Lemma 6]{Grong}. By a similar argument, $\widehat \pi(\mathcal O) = \widehat M.$ Both $\pi$ and $\widehat \pi$ are local diffeomorphisms. If $M$ and $\widehat M$ are simply connected, $\pi$ and $\widehat \pi$ will be diffeomorphisms. Define $\iota = \widehat \pi \circ \pi^{-1}.$
If $\pi(q)=x, q \in \mathcal O$, then $\iota_{*,x} = q$, so $\iota$ is an isometry.

\item[(c)] Follows from Lemma \ref{allndim}.
\end{itemize}
\end{proof}

\begin{corollary} \label{cor:Rolliso}
If $x(t)$ and $\widehat x(t)$ are two curves in a connected, simply connected manifold $M$ of constant sectional curvature. Then there is a rolling of $M$ on itself along $x(t)$ and $\widehat x(t)$ if and only if there is an orientation preserving isometry $\iota$ of $M$, so that $\iota (x(t)) = \widehat x(t)$.
\end{corollary}

\begin{proof}
If $\iota(x(t)) = \widehat x(t)$, then $t \mapsto \iota_{*,x(t)}$ is a rolling by the proof of Theorem \ref{theoremisometries}.

To prove the converse, assume that there is a rolling $q(t)$ such that $\pi(q(t)) = x(t)$ and $\widehat \pi(q(t)) = \widehat x(t)$. Then, from Theorem \ref{theoremisometries} (c) there is an isometry $\iota$, such that $q(t) = \iota_{*,x(t)}$, and the result follows.
\end{proof}

Another way of expressing the previous results is in terms of anti-development curves. Notice that from Lemma \ref{generalcondition}, we know that for a given curve $x(t)$ in $M$ with $y(t)$ as an anti-development curve, any other anti-development curve can only differ by an element in $\SO(n)$. Hence each curve has a unique equivalence class of anti-development curves $\SO(n) \cdot y(t)$ belonging to it. Then Corollary \ref{cor:Rolliso} can be reformulated in the following way.

\begin{corollary}
Two absolutely curves in a connected, simply connected Riemannian manifold of constant sectional curvature, have the same equivalence class of anti-development curves if and only if they differ by an isometry.
\end{corollary}

\section{Construction of a rolling motion from initial data}\label{sec:construction}
In this final section, we will make comments concerning the practical nature of constructing a concrete rolling motion along a given curve $x(t)$ in $M$ starting at $x_0$, with initial condition $q_0: T_{x_0} M \to T_{\widehat x_0} \widehat M$. In the simple case of rolling on $\real^n$, we can just find an anti-development as in Remark \ref{remarkantidev}. In general, we need to solve the differential equations \eqref{solve1} and \eqref{solve2}. Notice that unless $\widehat M$ has a local frame with constant Christoffel symbols, equations \eqref{solve1} and \eqref{solve2} are coupled, making them very difficult to solve in general, and even if we manage to make $\widehat \Gamma^k_{ij}$ constant, this still does not make equation \eqref{solve2} easy to solve.

However, given our new knowledge of the relationship between rolling and geodesic curvatures, we are able to give the following algorithm for constructing a rolling motion in the case that $x(t)$ is a $C^n$ curve that is $C^{n-1}$-regular.

\begin{itemize}

\item[(i)] Find the curvatures $\kappa_1(t), \dots, \kappa_{n-1}(t)$ and Frenet vector fields $v_1(t), \dots, v_n(t)$ of $x(t)$.

\item[(ii)] Find the curve $\widehat x(t)$ in $\widehat M$, with curvatures $\kappa_1(t), \dots, \kappa_{n-1}(t)$ and initial conditions $\widehat x(0) = \widehat x_0$, and with Frenet vector field alongs $\widehat v_1(t), \dots, \widehat v_n(t)$ satisfying
$\widehat v_j(0) = q_0 v_j(0).$

\item[(iii)] Finally, define $q(t)$, by $q(t) v_j(t) = \widehat v_j(t).$

\end{itemize}

The possible difficulty in solving this problems, lies in (ii), but even though finding this solution may be difficult, it does have the advantage that it only depends on information on $\widehat M$.
In explicit formulas, let $(\widehat \xi, \widehat U)$ be a chart on $\widehat M$ with a chosen positively oriented orthonormal basis of vector fields $\widehat e_1, \dots, \widehat e_n$ on $\widehat U$. Write $\widehat v_j(t) = \sum_{i=1}^n a_{ij}(t) \widehat e_i(\widehat x(t))$
and we define $a(t) = (a_{ij}(t))$ as a curve in $\SO(n)$, then
$$\dot{\widehat \xi}_i = \sum_{j=1}^n \widehat \phi_{ij} a_{j1},
\qquad \dot a = a K   - \sum_{s=1}^n a_{s1} \widehat \Gamma_s a,$$
where $\widehat \phi_{ij}$ and $\widehat \Gamma_s$ is defined as in the proof of Lemma \ref{lemma:Cn} and $K(t)$ is the antisymmetric tridiagonal matrix with zeros along the diagonal given by
$$K(t) := \begin{small} \left(\begin{array}{cccccc}
0 & -\kappa_1(t) & 0 & \cdots & 0 & 0 \\
\kappa_1(t) & 0 & -\kappa_2(t) &  \cdots & 0 & 0 \\
0 & \kappa_2(t) & 0 & \cdots & 0 & 0 \\
\vdots & \vdots & \vdots & \ddots & \vdots & \vdots \\
0 & 0 & 0 & \cdots &0 & - \kappa_n(t) \\
0 & 0 & 0 & \cdots & \kappa_n(t) & 0 \end{array} \right) \end{small}.$$
We illustrate this with some examples. We use $a_j$ to denote the $j-$th column vector of $a$, and ${}^\top$ to denote the transpose.

\begin{example} \label{1Examplefindcurve}
\begin{itemize}

\item[(a)] Let $\widehat M = \real^n$ with the Euclidean structure. Let us again use $(r_1, \dots, r_n)$ for the standard coordinates, and use $\partial_{r_j} = \widehat e_j$. Then $a(t)$ is a solution to the usual Frenet-Serret equation $\dot a = a K.$

\item[(b)] Let $\widehat M$ be $S^n$ with the usual metric and consider it as a subset of $\real^{n+1}$ with coordinates $(r_0, r_1, \dots, r_n)$. Assume that $\widehat x_0 \neq  (-1,0, \dots, 0)$. Then we can use the chart
$$(\widehat \xi_1, \dots, \widehat \xi_n) = \tfrac{1}{1+r_0}(r_1, \dots, r_n).$$
An orthonormal basis on $S^n \setminus \{(1, 0,\dots, 0)\}$ is given by $\widehat e_j = \tfrac{1+ |\xi|^2}{2} \partial_{\xi_j},$
where $|\xi|^2:= \sum_{j=1}^n \xi_j^2.$ The only nonzero Christoffel symbols are $\widehat \Gamma_{ij}^i = -\widehat \Gamma_{ii}^j = -\widehat \xi_j$ when $i \neq j,$
so the equations we must solve in order to find our curve is
$$\dot{\widehat  \xi} = \frac{2}{1+|\xi|^2} a_1, \qquad \dot a = a K + \left( \left( a_1\right)^\top \widehat \xi - \big(\widehat \xi\big)^\top a_1\right) a.$$

\item[(c)] A particular nice case is when $\widehat M = S^3$, where we have the advantage of being able to identify $S^3$ with the Lie group $\SU(2)$ of matrices
$$g = \left( \begin{array}{cc} g_0 + i g_1 & g_2 + i g_3 \\ -g_2 + i g_3 & g_0 - i g_1 \end{array} \right), \quad
\det g = 1.$$
Consider $\widehat x(t) = g(t)$ as a curve in these coordinates. The usual metric on $S^3$ is even bi-invariant with respect to the multiplication on $\SU(2)$.
We will choose the following positively oriented orthonormal basis that is also left invariant,
$$X_1 = - g_1 \partial_{g_0} + g_0 \partial_{g_1} + g_3 \partial_{g_2} - g_2 \partial_{g_3},$$
$$X_2 = - g_2 \partial_{g_0} - g_3 \partial_{g_1} + g_0 \partial_{g_2} + g_1 \partial_{g_3},$$
$$X_3 = - g_3 \partial_{g_0} + g_2 \partial_{g_1} - g_1 \partial_{g_2} + g_0 \partial_{g_3}.$$
We have the relations,
$$[X_1, X_2] = 2X_3, \quad [X_1, X_3] = - 2X_2, \quad [X_2, X_3] = 2X_1,$$
and from bi-invariance, we know that $\nabla_{X_i} X_j = \tfrac{1}{2} [X_i, X_j].$
The Christoffel symbols are hence constant in this basis, so there is no need for choosing local coordinates.
This reduces the final equation that needs to be solved to
$$\dot a = a \begin{small} \left( \begin{array}{ccc} 0 & -\kappa_1 & 0 \\ \kappa_1 & 0 & -(\kappa_2-1) \\ 0 & \kappa_2-1 & 0 \end{array} \right) \end{small}.$$
After solving this, we obtain the solution, by viewing $a_1(t)$ as a curve in the Lie algebra, and solving $\dot g(t) = g(t) \cdot a_1(t).$ On matrix form, this is written
$$\dot g(t) = g(t) \begin{small} \left( \begin{array}{cc} i a_{11} & a_{21} + i a_{31} \\ - a_{21} + i a_{31} & - i a_{11} \end{array}\right) \end{small}
$$
\end{itemize}
\end{example}

\end{document}